\documentclass[psamsfonts]{amsart}
\usepackage{latexsym, amssymb, amsmath}

\newtheorem{theorem}{Theorem}[section]
\newtheorem{lemma}[theorem]{Lemma}
\theoremstyle{definition}
\newtheorem{definition}[theorem]{Definition}
\newtheorem{corollary}[theorem]{Corollary}
\newtheorem{conjecture}[theorem]{Conjecture}
\newtheorem{remark}[theorem]{Remark}
\newtheorem{question}[theorem]{Question}
\newtheorem{notation}[theorem]{Notation}
\numberwithin{equation}{section}
\DeclareMathOperator{\sgn}{\alpha}
\DeclareMathOperator{\cyd}{\gamma}
\DeclareMathOperator{\clf}{\phi}
\newcommand{\Sp}{\mathbb{S}}
\newcommand{\norm}[1]{\|\,#1\,\|}
\newcommand{\ip}[2]{\left\langle#1,#2\right\rangle}
\newcommand{\conj}[1]{{#1}^*}
\newcommand{\Sum}[1]{\sum_{#1\in G}}

\newcommand{\sob}[1]{{\langle#1\rangle}}
\newcommand{\sbf}[1]{(-1)^{\langle#1\rangle}}
\def\urltilda{\kern -.15em\lower .7ex\hbox{\~{}}\kern .04em}                               
\makeindex
\begin{document}

\title{Cayley-Dickson and Clifford Algebras as Twisted Group Algebras (2003)}

\author{John W. Bales}
\date{}
\address{Department of Mathematics\\
         Tuskegee University\\
	 Tuskegee, AL 36088\\
	 USA}
\email{jbales@mytu.tuskegee.edu}


\begin{abstract} Given a finite group $G,$ a set of basis vectors $\mathcal{B}=\{i_p | p\in G\}$ and a `sign function' or `twist' $\sgn:G\times G\to\{-1,1\},$ there is a `twisted group algebra' defined on the set $V$ of all linear combinations of elements of $\mathcal{B}$ over a field $\mathfrak{F}$ such that if $p,q\in G,$ then $i_pi_q=\sgn(p,q)i_{pq}.$ This product is extended to $V$ by distribution. Examples of such twisted group algebras are the Cayley-Dickson algebras and Clifford algebras. It is conjectured that the Hilbert Space $\ell^2$ of square summable sequences is a Cayley-Dickson algebra.
\end{abstract}
\maketitle
\section{Introduction}
In 1972 I took a complex analysis course taught by William R. R. Transue using
a text by J. S. McNerney ``An Introduction to Analytic Functions with
Theoretical Implications''\cite{jM72}. An exercise in this book asked the student
to decide whether the scheme for multiplying two ordered pairs of real
numbers could be extended to higher dimensional spaces. At the time, I
was also taking a course taught by Coke Reed on the Hilbert space
$\ell^2$ of square summable sequences of real numbers. I decided to
investigate whether a product akin to the product of complex numbers could be extended in
a meaningful way to $\ell^2.$ My idea was to do this by equating an
ordered pair of sequences with the ``shuffling'' of the two sequences.
Being unaware at the time of the Cayley-Dickson construction I naively constructed a sequence of spaces
 utilizing the product
\begin{equation}
 (a,b)\cdot(c,d)=(ac-bd,ad+bc)
\end{equation}
This construction led to a sequence $i_0=1,i_1=i,i_2,i_3,\cdots$ of unit basis vectors for $\ell^2$ satisfying the twisted product
\begin{equation}
 i_p\cdot i_q=\eta(p,q)i_{pq}
\end{equation}
where $pq$ was the bit-wise ``exclusive or'' of the binary representations of $p$ and $q$ and
\begin{equation}
 \eta(p,q)=(-1)^{<p\wedge q>}
\end{equation}
where $p\wedge q$ is the bit-wise ``and'' function of $p$ and $q$ and $<r>$ represents the sum of the bits in the binary representation of $r$. Since the matrix associated with this function is a Hadamaard matrix, I called this the ``Hadamaard sign function.''
(I was unaware of the term `twist' and of twisted group algebras at the time since my training was in general topology, not algebra.)
Seeing that the four-dimensional `Hadamaard space' created by this construction was not the quaternions, but a space with zero divisors, I abandoned the project in 1972.
In 1983, while culling some of my papers, I came across the
notes I had made on this problem and began to look at it again. At that time I found a nice way to represent a 
twisted product using inner products and conjugates (Theorem \ref{T:product}). I gave a short talk on those results at a meeting of the Alabama Academy of Sciences in 1992 at Tuskegee University.

Last year I read John Baez' article `The Octonions' \cite{jB02} where I learned for the first time of twisted group algebras and of Cayley-Dickson and Clifford algebras. This motivated me to look at Cayley-Dickson and Clifford 
algebras. The resulting work is presented in this paper beginning with the section titled `The Cayley-Dickson Construction.' The sections prior to that section are the most general results 
of my work prior to 1985 not specifically tied to the Hadamaard twist. 

 \section{Twisted Group Algebras}
 
 Let $V$ denote an $n$-dimensional vector space over the field $\mathfrak{F}.$ Let $G$
 denote a group of order $n.$ Let $\mathcal{B}=\{i_p \mid p\in G \}$ denote a set of
 unit basis vectors for  $V.$ Then, for each $x\in V,$ there exist
 elements $\{x_p \mid x_p\in \mathfrak{F}, p\in G \}$ such that
 $x=\Sum{p}x_pi_p.$\\
 Define a product on the elements of $\mathcal{B}$ and their negatives in the following manner.
 \\[12pt]
 Let $\sgn: G \times G \mapsto \{ -1,1\}$ denote a \emph{sign} function on $G.$ Then for $p,q\in G$ define the product of $i_p$ and $i_q$ as follows.
 \begin{definition}
   \[i_p i_q = \sgn(p,q)i_{pq}\]
 \end{definition}
 Extend this product to $V$ in the natural way. That is,
 \begin{definition}
    \begin{align*}
    xy
    &= \left( \Sum{p}x_p i_p \right)\left(\Sum{q} y_q i_q \right)\\
    &= \Sum{p} \Sum{q} x_p y_q i_p  i_q\\
    &= \Sum{p} \Sum{q} x_p y_q \sgn(p,q) i_{pq}
    \end{align*}
 \end{definition}
 In defining the product this way, one gets the closure and distributive properties ``for free'', as well as $(cx)y=x(cy)=c(xy)$
 
 This product transforms the vector space $V$ into a \emph{twisted group algebra}. The properties of the algebra depend upon the properties of the twist and the properties of $G.$

 \begin{notation}
  Given a group $G$, twist $\alpha$ on $G$ and field $\mathfrak{F}$, let $[G,\alpha,\mathfrak{F}]$ denote the corresponding twisted group algebra. If $\mathfrak{F}=\mathbb{R}$, abbreviate this notation $[G,\alpha]$.
 \end{notation}
 
 \section{Twists and Field Properties}
 Let $G$ denote a finite group, $\mathfrak{F}$ a field and $\sgn$ a twist on $G$. Let $V=[G,\sgn,\mathfrak{F}].$
 \begin{definition}
  If $\sgn(p,e)=\sgn(e,p)=1$ where $e$ is the identity element of $G,$ then $\sgn$ is said to an \emph{identive} twist on $G.$
 \end{definition}
 \begin{theorem}\label{T:identity}
  If $\sgn$ is identive, then $i_e$ is the identity element, 1, of $V.$
 \end{theorem}
 \begin{proof}
  If $x\in V,$ then\\
  $x i_e=\left(\Sum{p}x_p i_p\right) i_e=\Sum{p}x_pi_pi_e=\Sum{p}x_p \sgn(p,e)i_{pe}=\Sum{p}x_p i_p=x.$
  The proof for $i_e x=x$ is similar.
  Thus, $i_e$ is the identity element, 1, of $V$
 \end{proof}
 \begin{definition}\label{D:associative}
  If $\sgn(p,q)\sgn(pq,r)=\sgn(p,qr)\sgn(q,r)$ for $p,q,r\in G,$ then $\sgn$ is said to be an \emph{associative} twist on $G.$
 \end{definition}
 \begin{theorem}\
  If $p,q,r\in G$ and $\sgn$ is associative, then $i_p \left(i_qi_r\right)=\left(i_p i_q\right)i_r.$
  \end{theorem}
\begin{proof}
  \begin{align*}
 i_p\left(i_qi_r\right)&=i_p\left(\sgn(q,r)i_{qr}\right)\\
                 &=\sgn(q,r)i_pi_{qr}=\sgn(p,qr)\sgn(q,r)i_{p(qr)}\\
                 &=\sgn(p,q)\sgn(pq,r)i_{(pq)r}\\
                 &=\sgn(p,q)i_{pq}i_r\\
                 &=\left(i_pi_q\right)i_r.
  \end{align*}
\end{proof}
\begin{theorem}\label{T:associative}
 If $x,y,z\in V$ and if $\sgn$ is associative, then
 $x\left(yz\right)=\left(xy\right)z.$
\end{theorem}
\begin{proof}
 $y z=\Sum{q}\Sum{r}\left(y_qz_r\right)i_qi_r,$ so
  \begin{align*}
 x(yz)&=
 \left(\Sum{p}x_pi_p\right)\left(\Sum{q}\Sum{r}\left(y_qz_r\right)i_qi_r\right)\\
 &=\Sum{p}\Sum{q}\Sum{r}x_p\left(y_qz_r\right)i_p\left(i_qi_r\right)\\
 &=\Sum{p}\Sum{q}\Sum{r}\left(x_py_q\right)z_r\left(i_pi_q\right)i_r\\
 &=\left(\Sum{p}\Sum{q}\left(x_py_q\right)i_pi_q\right)\Sum{r}z_ri_r\\
 &=\left(xy\right)z
  \end{align*}.
\end{proof}
\begin{definition}
 If $\sgn(e,e)=1,$ then $\sgn$ is said to be a \emph{positive} twist. Otherwise, $\sgn$ is a negative twist.
 \end{definition}
 \begin{theorem}
  If $\sgn$ is associative, then $\sgn(e,p)=\sgn(p,e)=\sgn(e,e).$
 \end{theorem}
 \begin{proof}
   $\sgn(e,e)\sgn(e,p)=\sgn(e,e)\sgn(ee,p)\\
   =\sgn(e,ep)\sgn(e,p)=\sgn(e,p)\sgn(e,p)=1.$\\
   Thus $\sgn(e,p)=\sgn(e,e).$\\
   $\sgn(p,e)\sgn(e,e)=\sgn(p,ee)\sgn(e,e)\\
   =\sgn(p,e)\sgn(pe,e)=\sgn(p,e)\sgn(p,e)=1.$\\
   Thus $\sgn(p,e)=\sgn(e,e)$
 \end{proof}
 \begin{corollary}\label{C:identive}
  Every positive associative twist is identive.
 \end{corollary}
\begin{theorem}\label{T:altAssoc}
  The twist $\sgn$ is associative if and only if, for $p,q,r\in G$ $\sgn(p,q)\sgn(q,r)=\sgn(p,qr)\sgn(pq,r)$
 \end{theorem}
 \begin{proof}
  Multiplying each side of the equation \\
  $\sgn(p,q)\sgn(pq,r)=\sgn(p,qr)\sgn(q,r)$ by $\sgn(q,r)\sgn(pq,r)$ yields\\
  $\sgn(p,q)\sgn(q,r)=\sgn(p,qr)\sgn(pq,r).$ So the two conditions are equivalent.
 \end{proof}
 \begin{theorem}
 If $\sgn$ is positive and associative, then $<V,+,\cdot>$ is a ring with unity.
 \end{theorem}
 \begin{proof}
  Follows immediately from Theorems~\ref{T:identity} and \ref{T:associative} and Corollary~\ref{C:identive}.
 \end{proof}

\begin{theorem}\label{T:lr_inverse}
  For each $p\in G,$ $\sgn\left(p,p^{-1}\right)i_{p^{-1}}$ and $\sgn\left(p^{-1},p\right)i_{p^{-1}}$
  are right and left inverses, respectively, of $i_p.$
 \end{theorem}
 \begin{proof}
  \begin{align*}
  i_p\left(\sgn\left(p,p^{-1}\right)i_{p^{-1}}\right)
    &=\sgn\left(p,p^{-1}\right)i_pi_{p^{-1}}\\
    &=\sgn\left(p,p^{-1}\right)\left(\sgn\left(p,p^{-1}\right)i_{pp^{-1}}\right)\\
    &=i_e=1\\
  \left(\sgn\left(p^{-1},p\right)i_{p^{-1}}\right)i_p
    &=\sgn\left(p^{-1},p\right)i_{p^{-1}}i_p\\
    &=\sgn\left(p^{-1},p\right)\left(\sgn\left(p^{-1},p\right)i_{pp^{-1}}\right)\\
    &=i_e=1
  \end{align*}
 \end{proof}
 \begin{definition}\label{D:invertive}
  If $\sgn\left(p,p^{-1}\right)=\sgn\left(p^{-1},p\right)$ for $p\in G,$ then $\sgn$ is an \emph{invertive} twist on $G.$
   \end{definition}
 \begin{theorem}
  If $p\in G$ and if $\sgn$ is invertive, then $i_p$ has an inverse \\
  $i_p^{-1}=\sgn\left(p,p^{-1}\right)i_{p^{-1}}=\sgn\left(p^{-1},p\right)i_{p^{-1}}.$
 \end{theorem}
 \begin{proof}
  Follows immediately from Theorem \ref{T:lr_inverse} and Definition \ref{D:invertive}.
 \end{proof}
 \begin{definition}\label{D:conjugate}
  If $\sgn$ is invertive, and $x\in V,$ then let $\conj{x}=\Sum{p}\conj{x}_pi_p^{-1}$ denote
  the \emph{conjugate} of $x.$
 \end{definition}
 \begin{theorem}\label{T:conjugate}
  If $\sgn$ is an invertive twist on $G,$ and if $x,y\in V,$ then
  \begin{enumerate}
   \item[(i)] $\conj{x}=\Sum{p}\sgn\left(p^{-1},p\right)\conj{x}_{p^{-1}}i_p$
   \item[(ii)] $\conj{\conj{x}}=x$
   \item[(iii)] $\conj{(x+y)}=\conj{x}+\conj{y}$
   \item[(iv)] $\conj{(cx)}=\conj{c}\conj{x}$ for all $c\in \mathfrak{F}.$
  \end{enumerate}
 \end{theorem}
  \begin{proof}
   \begin{enumerate}
    \item[]
    \item[(i)] Let $q^{-1}=p.$ Then 
     \begin{align*}
       \conj{x}&=\Sum{q}\conj{x}_qi_q^{-1}\\
                &=\Sum{q}\conj{x}_q\sgn\left(q,q^{-1}\right)i_{q^{-1}}\\
                &=\Sum{p}\sgn\left(p^{-1},p\right)\conj{x}_{p^{-1}}i_p
     \end{align*}
    \item[(ii)] Let $z=\conj{x}.$ Then
     \begin{align*} z&=\Sum{p}\sgn\left(p^{-1},p\right)\conj{x}_{p^{-1}}i_p\\
                           &=\Sum{p}z_pi_p
       \end{align*}
     where $z_p=\sgn\left(p^{-1},p\right)\conj{x}_{p^{-1}}.$ Then $z_{p^{-1}}=\sgn\left(p,p^{-1}\right)\conj{x}_p,$ and
     \begin{align*}
     \conj{z}_{p^{-1}}&=\sgn\left(p,p^{-1}\right)x_p\text{.\ So}\\ 
     \conj{\conj{x}}&=\conj{z}\\
              &=\Sum{p}\sgn\left(p^{-1},p\right)\conj{z}_{p^{-1}}i_p\\
              &=\Sum{p}\sgn\left(p^{-1},p\right)\sgn\left(p,p^{-1}\right)x_pi_p \\
              &=\Sum{p}x_pi_p=x
       \end{align*}
    \item[(iii)] 
     \begin{align*}
        \conj{\left(x+y\right)}&=\Sum{p}\conj{\left(x_p+y_p\right)}i_p^{-1}\\
                                    &=\Sum{p}\left(\conj{x}_p+\conj{y}_p\right)i_p^{-1}\\
                                    &=\Sum{p}\conj{x}_pi_p^{-1} +\Sum{p}\conj{y}_pi_p^{-1}\\
                                    &=\conj{x}+\conj{y}.
       \end{align*}
    \item[(iv)] 
     \begin{align*}
        \conj{(cx)}&=\Sum{p}\conj{(cx_p)}i_p^{-1}\\
                  &=\Sum{p}\conj{c}\conj{x}_pi_p^{-1}\\
                  &=\conj{c}\Sum{p}\conj{x}_pi_p^{-1}\\
                  &=\conj{c}\conj{x}
       \end{align*}
   \end{enumerate}
  \end{proof}
  
\section{Proper Sign Functions}
  \begin{definition}\label{D:proper}
   The statement that the twist $\sgn$ on $G$ is  \emph{proper} means that if $p,q\in G,$ then
   \begin{enumerate}
    \item[(1)] $\sgn(p,q)\sgn\left(q,q^{-1}\right)=\sgn\left(pq,q^{-1}\right)$
    \item[(2)] $\sgn\left(p^{-1},p\right)\sgn(p,q)=\sgn\left(p^{-1},pq\right).$
   \end{enumerate}
  \end{definition}
  \begin{theorem}\label{T:assocprod}
   Every positive associative twist is proper.
  \end{theorem}
  \begin{proof}\qquad\\
    \begin{enumerate}
     \item[(1)] $\sgn(p,q)\sgn\left(q,q^{-1}\right)
                 =\sgn\left(pq,q^{-1}\right)\sgn\left(p,qq^{-1}\right)\\
                 =\sgn\left(pq,q^{-1}\right)\sgn(p,e)=\sgn\left(pq,q^{-1}\right)\sgn(e,e)
                 =\sgn\left(pq,q^{-1}\right)$
     \item[(2)] $\sgn\left(p^{-1},p\right)\sgn(p,q)
                 =\sgn\left(p^{-1}p,q\right)\sgn\left(p^{-1},pq\right)\\
                 =\sgn(e,q)\sgn\left(p^{-1},pq\right)=\sgn(e,e)\sgn\left(p^{-1},pq\right)
                 =\sgn\left(p^{-1},pq\right)$
     \end{enumerate}
  \end{proof}
  \begin{theorem}\label{T:conjid}
   Every proper twist is positive, identive and invertive.
  \end{theorem}
  \begin{proof}
    Suppose $\sgn$ is proper. \\
    Then $\sgn(p,e)\sgn\left(e,e^{-1}\right)=\sgn\left(pe,e^{-1}\right).$ That is, $\sgn(p,e)\sgn(e,e)=\sgn(p,e).$ So, $\sgn(e,e)=1.$ Thus, $\sgn$ is positive.\\
    Furthermore, $\sgn(e,q)\sgn\left(q,q^{-1}\right)=\sgn\left(eq,q^{-1}\right)=\sgn\left(q,q^{-1}\right).$ So, $\sgn(e,q)=1.$\\
    Also, $\sgn\left(p^{-1},p\right)\sgn(p,e)=\sgn\left(p^{-1},pe\right)=\sgn\left(p^{-1},p\right).$ So $\sgn(p,e)=1.$ Thus, $\sgn$ is identive.\\
    And since $\sgn\left(p,p^{-1}\right)\sgn\left(p^{-1},p\right)=\sgn\left(pp^{-1},p\right)=\sgn(e,p)=1,$ it follows that $\sgn\left(p,p^{-1}\right)=\sgn\left(p^{-1},p\right).$ So $\sgn$ is invertive.
  \end{proof}

  \begin{theorem}
  If $\sgn$ is a proper twist on $G,$ then
   $\conj{(i_pi_q)}=\conj{i_q}\conj{i_p}$ for all $p,q\in G.$
  \end{theorem}
  \begin{proof}
 Since $i_pi_q=\sgn(p,q)i_{pq},$ $\conj{\left(i_pi_q\right)}
   =\conj{\left(\sgn(p,q)i_{pq}\right)}=\sgn(p,q)\conj{\left(i_{pq}\right)}\\
   =\sgn(p,q)i_{pq}^{-1} =\sgn(p,q)\sgn\left(\left(pq\right)^{-1},pq\right)i_{(pq)^{-1}}.$\\
   On the other hand,  
   \begin{align*}
   \conj{i_q}\conj{i_p}=i_q^{-1}i_p^{-1}&=
   \left(\sgn\left(q^{-1},q\right)i_{q^{-1}}\right)\left(\sgn\left(p^{-1},p\right)i_{p^{-1}}\right)\\
   &=\sgn\left(q^{-1},q\right)\sgn\left(p^{-1},p\right)\sgn\left(q^{-1},p^{-1}\right)i_{q^{-1}p^{-1}}\\
   &=\sgn\left(q^{-1},q\right)\sgn\left(p^{-1},p\right)\sgn\left(q^{-1},p^{-1}\right)i_{(pq)^{-1}}
   \end{align*}
   Therefore, in order to show that $\conj{(i_pi_q)}
   =\conj{i_q}\conj{i_p},$ it is sufficient to show that\\
   $\sgn(p,q)\sgn\left(\left(pq\right)^{-1},pq\right)=\sgn\left(q^{-1},q\right)\sgn\left(p^{-1},p\right)\sgn\left(q^{-1},p^{-1}\right).$\\
   Beginning with the expression on the left,
   \begin{align*}
   \sgn(p,q)\sgn\left((pq)^{-1},pq\right) 
&=\sgn(p,q)\sgn\left(q,q^{-1}\right)\sgn\left(q,q^{-1}\right)\sgn\left((pq)^{-1},pq\right)\\
&=\sgn\left(pq,q^{-1}\right)\sgn\left(q,q^{-1}\right)\sgn\left((pq)^{-1},pq\right)\\
&=\sgn\left((pq)^{-1},pq\right)\sgn\left(pq,q^{-1}\right)\sgn\left(q,q^{-1}\right)\\
&=\sgn\left((pq)^{-1},p\right)\sgn\left(q,q^{-1}\right)\\
&=\sgn\left((pq)^{-1},p\right)\sgn\left(p,p^{-1}\right)\sgn\left(p,p^{-1}\right)\sgn\left(q,q^{-1}\right)\\
&=\sgn\left(q^{-1}p^{-1},p\right)\sgn\left(p,p^{-1}\right)\sgn\left(p,p^{-1}\right)\sgn\left(q,q^{-1}\right)\\
&=\sgn\left(q^{-1},p^{-1}\right)\sgn\left(p,p^{-1}\right)\sgn\left(q,q^{-1}\right)\\
&=\sgn\left(q^{-1},p^{-1}\right)\sgn\left(p^{-1},p\right)\sgn\left(q^{-1},q\right)
   \end{align*}
  \end{proof}
  \begin{theorem}
  If $\sgn$ is a proper twist on $G$ and if $x,y\in V,$ then
  $\conj{(xy)}=\conj{y}\conj{x}.$
  \end{theorem}
  \begin{proof}
   \begin{align*}
     \conj{(xy)}&=\conj{\left(\left(\Sum{p}x_pi_p\right)\left(\Sum{q}y_qi_q\right)\right)}\\
                     &=\conj{\left(\Sum{p}\Sum{q}x_py_qi_pi_q\right)}\\
                     &=\Sum{p}\Sum{q}\conj{(x_py_qi_pi_q)}\\
                     &=\Sum{q}\Sum{p}\conj{y_q}\,\conj{x_p}i_q^{-1}i_p^{-1}\\
                     &=\left(\Sum{q}\conj{y_q}i_q^{-1}\right)\left(\Sum{p}\conj{x_p}i_p^{-1}\right)\\
                     &=\conj{y}\conj{x}
   \end{align*}
  \end{proof}
  \begin{definition}
   The \emph{inner product} of elements $x$ and  $y$ in $V$ is $\ip{x}{y}
   =\Sum{p}x_p\conj{y}_p.$
  \end{definition}
  \begin{theorem}\label{T:zero}
   $\left[1-\sgn(p,p)\right]\ip{x}{i_px}=0$ provided $\sgn$ is proper and $p=p^{-1}.$
  \end{theorem}
  \begin{proof}
   $i_px=\sum_qx_qi_pi_q=\sum_q\sgn(p,q)x_qi_{pq}=\sum_r\sgn(p,pr)x_{pr}i_r$ where $q=pr.$
   
   \begin{align*}
   2\ip{x}{i_px}&=2\sum_r\sgn(p,pr)x_rx_{pr}\\
                &=2\,\sgn(p,p)\sum_r\sgn(p,r)x_rx_{pr}\\
		&=\sgn(p,p)\sum_r\sgn(p,r)x_rx_{pr}+\sgn(p,p)\sum_q\sgn(p,q)x_qx_{pq}\\
		&=\sgn(p,p)\sum_r\sgn(p,r)x_rx_{pr}+\sgn(p,p)\sum_r\sgn(p,pr)x_{pr}x_r\\
		&=\sgn(p,p)\sum_r\sgn(p,r)x_rx_{pr}+\sgn(p,p)\sgn(p,p)\sum_r\sgn(p,r)x_rx_{pr}\\
		&=\left[\sgn(p,p)+1\right]\sum_r\sgn(p,r)x_rx_{pr}\\
		&=\left[\sgn(p,p)+1\right]\sgn(p,p)\sum_r\sgn(p,pr)x_rx_{pr}\\
		&=\left[1+\sgn(p,p)\right]\ip{x}{i_px}
   \end{align*}
   Thus $\left[1-\sgn(p,p)\right]\ip{x}{i_px}=0.$
  \end{proof}
  \begin{corollary}
   If $p=p^{-1},$ $\sgn$ is proper and $\sgn(p,p)=-1,$ and if $x\in V$, then $\ip{x}{i_px}=0.$
  \end{corollary}
  \begin{theorem}\label{T:zero1}
   $\left[1-\sgn(p,p)\right]\ip{x}{x\,i_p}=0$ provided $\sgn$ is proper and $p=p^{-1}$
  \end{theorem}
 \begin{proof}
   The proof is similar to that of Theorem \ref{T:zero}.
 \end{proof}
  \begin{theorem}\label{T:product}
   If $\sgn$ is a proper twist on $G,$ and if $x,y\in V,$ then
   \[ xy=\Sum{r}\ip{x}{i_r\conj{y}}i_r =\Sum{r}\ip{y}{\conj{x}i_r}i_r
   \]
  \end{theorem}
  \begin{proof}
  ~\\
  (i) $\conj{y}
  =\Sum{s}\sgn\left(s,s^{-1}\right)\conj{y}_{s^{-1}}i_s.$ 
  Let $p=rs.$ Then
  \begin{align*} 
  i_r\conj{y}& =\Sum{s}\sgn\left(s,s^{-1}\right)\conj{y}_{s^{-1}}i_r i_s\\
  &=\Sum{s}\sgn(r,s)\sgn\left(s,s^{-1}\right)\conj{y}_{s^{-1}}i_{rs}\\
  &=\Sum{s}\sgn\left(rs,s^{-1}\right)\conj{y}_{s^{-1}}i_{rs}\\
  &=\Sum{p}\sgn\left(p,p^{-1}r\right)\conj{y}_{p^{-1}r}i_p
  \end{align*}
 Thus,
\begin{align*}
   \ip{x}{i_r \conj{y}}
    &=\Sum{p}\left(p,p^{-1}r\right)x_p\conj{\conj{y}}_{p^{-1}r}\\
    &=\Sum{p}\sgn\left(p,p^{-1}r\right)x_py_{p^{-1}r}
  \end{align*}
  (ii) $\conj{x}=\Sum{s}\sgn\left(s^{-1},s\right)\conj{x}_{s^{-1}}i_s.$ Let $q=sr.$ Then
   \begin{align*}
    \conj{x}i_r& =\Sum{s}\sgn\left(s^{-1},s\right)\conj{x}_{s^{-1}}i_si_r\\
                      & =\Sum{s}\sgn\left(s^{-1},s\right)\sgn(s,r)\conj{x}_{s^{-1}}i_{sr}\\
                      & =\Sum{s}\sgn\left(s^{-1},sr\right)\conj{x}_{s^{-1}}i_{sr}\\
                      & = \Sum{q}\sgn\left(rq^{-1},q\right)\conj{x}_{rq^{-1}}i_q
   \end{align*} 
   So $\ip{ y}{\conj{x}i_r } = \Sum{q}\sgn\left(rq^{-1},q\right)x_{rq^{-1}}y_q.$\\
  (iii) 
   \begin{align*}
     xy &= \Sum{p}\Sum{q}x_py_qi_pi_q\\
                  &= \Sum{p}\Sum{q}\sgn(p,q)x_py_qi_{pq}\\
                  &= \Sum{q}\Sum{p}\sgn(p,q)x_py_qi_{pq}
   \end{align*} 
Let $pq=r.$ Then $p=rq^{-1}$ and $q=p^{-1}r.$\\
   Thus,\\
   $xy=\Sum{r}\Sum{q}\sgn\left(rq^{-1},q\right)x_{rq^{-1}}y_qi_r=\Sum{r}\ip{y}{\conj{x}i_r}i_r.$\\
   And\\
   $xy=\Sum{r}\Sum{p}\sgn\left(p,p^{-1}r\right)x_py_{p^{-1}r}i_r=\Sum{r}\ip{x}{i_r\conj{y}}i_r.$
  \end{proof}
\begin{corollary}
 If $\sgn$ is proper, then $\ip{xy}{i_r}=\ip{x}{i_r\conj{y}}=\ip{\conj{x}i_r}{y}.$
\end{corollary}

 The twists on a group G form an abelian group $\mathfrak{S}(G),$ with the associative and proper twists forming closed subgroups $\mathfrak{S}_A(G)$ and $\mathfrak{S}_P(G)$ respectively.
  
  \begin{question}
   Do $\mathfrak{S}_A(G)$ and $\mathfrak{S}_P(G)$ have any interesting properties?
  \end{question}
  
 \section{Addendum on Associative Sign Functions}

\begin{theorem}  
   Suppose $\sgn$ is an associative twist on the group $G.$ For each $p\in G,$ let $L_p$ denote a matrix whose rows and columns are indexed by $G,$ such that if $L_p=\left(l_{rs}\right),$ 
 then $l_{rs}=\begin{cases}
                \sgn(p,s) & \text{\ if\ } r=ps\\
		0         & \text{\ otherwise}
              \end{cases}$

  Then for every $p,q\in G$, $L_pL_q=\sgn(p,q)L_{pq}.$	      
\end{theorem}

\begin{proof}
 Let $L_p=A=\left(a_{rk}\right)$ and $L_q=B=\left(b_{ks}\right).$
 
 then $a_{rk}=\begin{cases}
                \sgn(p,k) & \text{\ if\ } r=pk\\
		0         & \text{\ otherwise}
              \end{cases}$
	      
 and $b_{ks}=\begin{cases}
                \sgn(q,s) & \text{\ if\ } k=qs\\
		0         & \text{\ otherwise}
              \end{cases}.$
	      
Thus 
\begin{align*}
         L_pL_q = AB & = \left( a_{rk}\right)\left( b_{ks}\right)\\
	             & = \begin{cases}
                          \sgn(p,k)\sgn(q,s) & \text{\ if\ } r=pk \text{\ and\ } k=qs\\
		          0         & \text{\ otherwise}
                        \end{cases}\\
	             & = \begin{cases}
                          \sgn(p,qs)\sgn(q,s) & \text{\ if\ } r=pqs\\
		          0         & \text{\ otherwise}
                        \end{cases}\\
	             & = \begin{cases}
                          \sgn(p,q)\sgn(pq,s) & \text{\ if\ } r=pqs\\
		          0         & \text{\ otherwise}
                        \end{cases}\\
	             & = \sgn(p,q)L_{pq}
\end{align*}
	      
\end{proof}
\begin{corollary}
 If $\sgn$ is associative on $G,$ and $U$ is the set of all linear combinations of $\left\{~L_p~|~p\in G\right\}$ over $\mathfrak{F}$, then $U$ is isomorphic to $[G,\sgn,\mathfrak{F}].$
\end{corollary}

\section{The Cayley-Dickson Construction}

  This and the following sections contain examples of two sign function spaces or `twisted group algebras': the Cayley-Dickson spaces and the Clifford Algebras.
   
Let $\Sp_0$ denote the real numbers. Define
$\Sp_{n+1}=\Sp_n\times\Sp_n.$ Given $(a,b)$ and $(c,d)$ in $\Sp_{n+1},$
define their product as 
\[(a,b)(c,d)=(ac-db^*,a^*d+cb)\]
 and their
conjugate as 
\[\conj{(a,b)}=(\conj{a},-b)\]
This is the Cayley-Dickson construction. It defines an infinite sequence
of algebras $\Sp_0,\Sp_1,\Sp_2,\cdots,$ with$\Sp_0=\mathbb{R}$ being the reals,  $\Sp_1=\mathbb{C}$ being the complex numbers, $\Sp_2=\mathbb{H}$ the quaternions, $\Sp_3=\mathbb{O}$ the octonions, $\Sp_4$ the
sedenions, etc.

\section{Ordered Pairs as Shuffled Sequences}

 Let $x=x_0,x_1,x_2,\cdots\in\ell^2$ and
$y=y_0,y_2,y_3,\cdots\in\ell^2.$ Let the ordered pair $(x,y)$ denote the
``shuffled'' sequence $x_0,y_0,x_1,y_1,\cdots\in\ell^2.$
 Define the conjugate of $x$ as $x^*=x_0,-x_1,-x_2,\cdots.$ Then it
immediately follows that 
\[(x,y)^*=(x^*,-y)\]
 Equate a real number $\alpha$ with the sequence
$\alpha,0,0,0,\cdots=(\alpha,0).$
The unit basis vectors of the space $\ell^2$ are
  \begin{itemize}
   \item $i_0=1,0,0,\cdots=(1,0)=(i_0,0)=1$
   \item $i_1=0,1,0,0,\cdots=(0,1)=(0,i_0)=i$
   \item $i_2=0,0,1,0,\cdots=(i_1,0)=j$
   \item $i_3=0,0,0,1,0,\cdots=(0,i_1)=k$
   \item $i_{2n}=(i_n,0)$ for $n\ge 0$
   \item $i_{2n+1}=(0,i_n)$ for $n\ge 0$   
  \end{itemize}
Note that this scheme for numbering the basis vectors for the
Cayley-Dickson algebras produces a different numbering from those in
common use for octonions and upwards. Yet it arises naturally from the
equivalence of ordered pairs and shuffled sequences.

\section{The Sign Function and Product of Basis Vectors}

 Let $G$ denote the set of whole numbers $\{0,1,2,3,\cdots\}.$ 

 Let $G_n=\left\{p \,|\, 0\le p < 2^n\right\}.$ 

If $p\in G,$ $2p$ will denote twice the
value of $p.$ But for literals $p$ and $q,$ $pq$ will denote the
bit-wise `exclusive or' of the binary representations of $p$ and $q.$

For example, if $p=9$ and $q=11,$ the binary representations are
$p=1001$ and $q=1011.$ Applying the `exclusive or' operation to the
corresponding bits of the two numbers yields $pq=0010.$ Thus, $pq=2.$
The set $G$ is a group under this operation, with
identity element 0. Furthermore, $p^2=0$ for all 
$p\in G.$ Thus, $p=p^{-1}$ for each element $p$ in the group. $G_n$ is a subgroup of $G$ for each $n\ge 0.$ $G_n=\mathbb{Z}_2\times\mathbb{Z}_2\times\cdots\times\mathbb{Z}_2=\mathbb{Z}_2^n$ is the direct product of $n$ copies of the cyclic 2-group $\mathbb{Z}_2$, also known as the \emph{dyadic} group of order $n.$ The group operation 
satisfies the following properties:
\begin{enumerate}
 \item $(2p)(2q)=2pq$
 \item $(2p)(2q+1)=2pq+1$
 \item $(2p+1)(2q)=2pq+1$
 \item $(2p+1)(2q+1)=2pq$
\end{enumerate}
The Cayley-Dickson product of unit vectors satisfy the following
  \begin{enumerate}
   \item $i_{2p}i_{2q}=(i_p,0)(i_q,0)=(i_pi_q,0)$
   \item $i_{2p}i_{2q+1}=(i_p,0)(0,i_q)=(0,i_p^*i_q)$
   \item $i_{2p+1}i_{2q}=(0,i_p)(i_q,0)=(0,i_qi_p)$
   \item $i_{2p+1}i_{2q+1}=(0,i_p)(0,i_q)=-(i_qi_p^*,0)$
  \end{enumerate}
\begin{theorem}\label{T:sign}
There is a twist $\cyd(p,q)$ mapping $G\times G$ into
$\{-1,1\}$ such that if $p,q\in G,$ then
$i_pi_q=\cyd(p,q)i_{pq}.$
\end{theorem}
\begin{proof}
 
Assume $0\le p < 2^n$ and $0\le q < 2^n$ and proceed by induction on
$n.$

If $n=0,$ then $p=q=0$ and $i_pi_q=i_0i_0=i_0=\cyd(p,q)i_{pq}$ provided
$\cyd(0,0)=1.$

Suppose the principle is true for $n=k.$ Let $n=k+1.$ Let $0\le p < 2^n$
and $0\le q < 2^n.$ Then there are numbers $r$ and $s$ such that $0\le r
< 2^k$ and $0\le s < 2^k$ and such that one of the following is true:
\begin{itemize}
 \item $p=2r,$ $q=2s$
 \item $p=2r,$ $q=2s+1$
 \item $p=2r+1,$ $q=2s$
 \item $p=2r+1,$ $q=2s+1$
\end{itemize}
\begin{enumerate}
  \item  Assume $p=2r,$ $q=2s.$ Then
                \begin{align*} i_pi_q    & = i_{2r}i_{2s}        = (i_ri_s,0)\\
                                         & = (\cyd(r,s)i_{rs},0) =\cyd(r,s)(i_{rs},0)\\
                                         & = \cyd(r,s)i_{2rs}    = \cyd(2r,2s)i_{(2r)(2s)}\\
                                         & = \cyd(p,q)i_{pq}
                  \end{align*}           
                provided $\cyd(2r,2s)=\cyd(r,s).$

  \item Assume $p=2r,$ $q=2s+1.$ Then 
        $i_pi_q  =i_{2r}i_{2s+1}=(0,i_r^*i_s).$
                 
                 If $r\ne 0,$ then           
                   \begin{align*} i_pi_q & =-(0,i_ri_s)=-(0,\cyd(r,s)i_{rs})\\
                                         & =-\cyd(r,s)i_{2rs+1}=\cyd(2r,2s+1)i_{(2r)(2s+1)}\\
                                         & = \cyd(p,q)i_{pq}
                    \end{align*}
               
                provided $\cyd(2r,2s+1)=-\cyd(r,s)$ when $r\ne 0.$
               
               If $r=0,$ then \begin{align*}i_pi_q &=i_0i_{2s+1}=(0,i_0i_s)\\
                                                   &=(0,\cyd(0,s)i_s)=\cyd(0,s)i_{2s+1}\\
                                                   &=\cyd(0,2s+1)i_{pq}=\cyd(p,q)i_{pq}
               \end{align*} provided $\cyd(0,2s+1)=\cyd(0,s).$
  \item  Assume $p=2r+1,$ $q=2s.$ Then
                  \begin{align*}
                     i_pi_q &= i_{2r+1}i_{2s} = (0,i_si_r)\\
                            &= \cyd(s,r)(0,i_{sr}) = \cyd(s,r)i_{2sr+1}\\
                            &= \cyd(2r+1,2s)i_{(2r+1)(2s)}=\cyd(p,q)i_{pq}
                   \end{align*} provided $\cyd(2r+1,2s)=\cyd(s,r).$
  \item   Assume $p=2r+1,$ $q=2s+1.$ Then
          $i_pi_q=i_{2r+1}i_{2s+1}=-(i_si_r^*,0).$ 
                  If $r\ne 0,$ then
                 \begin{align*}
                   i_pi_q &= (i_si_r,0)=\cyd(s,r)(i_{sr},0)\\
                          &= \cyd(s,r)i_{2sr}=\cyd(2r+1,2s+1)i_{(2r+1)(2s+1)}\\
                          &= \cyd(p,q)i_{pq}
                 \end{align*} provided $\cyd(2r+1,2s+1)=\cyd(s,r)$ when $r\ne 0.$\\
                 If $r=0,$ then \begin{align*}
                                 i_pi_q &= i_1i_{2s+1}=-(i_si_0^*,0)\\
                                        &= -(i_si_0,0) = -\cyd(s,0)(i_s,0)\\
                                        &= -\cyd(s,0)i_{2s}=\cyd(1,2s+1)i_{1(2s+1)}\\
                                        &= \cyd(p,q)i_{pq}
                                \end{align*} provided $\cyd(1,2s+1)=-\cyd(s,0).$  
  
\end{enumerate}
Thus, the principle is true for $n=k+1$ provided the twist is
defined as required in these four cases.
\end{proof}

\begin{corollary}\label{C:sign} The requisite properties of the sign
function are
 \begin{enumerate}
  \item $\cyd(0,0)=1$     
  \item $\cyd(2r,2s)=\cyd(r,s)$     
  \item If $r\ne 0$ then $\cyd(2r,2s+1)=-\cyd(r,s)$       
  \item $\cyd(0,2s+1)=\cyd(0,s)=1$     
  \item $\cyd(2r+1,2s)=\cyd(s,r)$    
  \item If $r\ne 0$ then $\cyd(2r+1,2s+1)=\cyd(s,r)$    
  \item $\cyd(1,2s+1)=-\cyd(s,0)=-1$    
 \end{enumerate}
\end{corollary}

Let us apply Theorem \ref{T:sign} and Corollary \ref{C:sign} to the
example of finding the Cayley-Dickson product of the basis vectors
$i_{9}$ and $i_{11}.$ The process is easier if 9 and 11 are written in their
binary representations 1001 and 1011. Their product under the `bit-wise
exclusive or' group operation is 0010, or 2. Thus,
$i_9i_7=\cyd(1001,1011)i_{0010}.$ The twist can be worked out
using Corollary \ref{C:sign} as follows:
\begin{align*}
  \cyd(1001,1011) &= \cyd(101,100) &&\text{(by Corollary\ref{C:sign}.6)}\\
                  &= \cyd(10,10)   &&\text{(by Corollary\ref{C:sign}.5)}\\
                  &= \cyd(1,1)      &&\text{(by Corollary \ref{C:sign}.1)}\\
                  &= -1             &&\text{(by Corollary \ref{C:sign}.7)}\\
\end{align*}
Thus, $i_9i_{11}=-i_{2}.$

The following establishes the ``quaternion properties'' of the twist
\begin{theorem}\label{T:quaternion} If $0\ne p\ne q\ne 0$ then
 \begin{enumerate}
  \item $\cyd(p,p)=-1$
  \item $\cyd(p,q)=-\cyd(q,p)$
  \item $\cyd(p,q)=\cyd(q,pq)=\cyd(pq,p)$
 \end{enumerate}
\end{theorem}
\begin{proof}
  These follow by induction from Corollary \ref{C:sign}.
\end{proof}
\begin{theorem}\label{T:proper}
 If $p,q\in G,$ then
 \begin{enumerate}
  \item $\cyd(p,q)\cyd(q,q)=\cyd(pq,q)$
  \item $\cyd(p,p)\cyd(p,q)=\cyd(p,pq)$
 \end{enumerate}
 Thus $\cyd$ is a proper twist.
\end{theorem}
\begin{proof}
 If either $p$ or $q$ is 0, or if $p=q,$ the results are immediate from
Corollary \ref{C:sign}. If $0\ne p\ne q\ne 0,$ the results follow from
Theorem \ref{T:quaternion}.
\end{proof}
Since $\cyd$ is proper, Theorem \ref{T:product} applies. That is, if $x$ and $y$ are in $\Sp_n,$ then 
                                   \begin{align*}
                                     xy &= \sum_r\ip{x}{i_r\conj{y}}i_r\\
                                        &= \sum_r\ip{\conj{x}i_r}{y}i_r
                                   \end{align*}
Note that, since $\cyd(p,p)=-1$ for $p\ne0,$ the Cayley-Dickson conjugate is equivalent to the conjugate defined in Definition \ref{D:conjugate}.\\

\section{Is $\ell^2$ a Cayley-Dickson Algebra?}

My original goal was to extend the Cayley-Dickson product to the space $\ell^2$ of square summable sequences. If, for elements $x,y\in\ell^2$, the product $xy$ is defined as above, then $xy$ is a number sequence since $\ip{x}{i_r\conj{y}}<\infty$ for each $r$, but $xy$ is not obviously square-summable. 

Since the components of the product $xy$ are $\ip{x}{i_r\conj{y}}$, if $i_r\conj{y}$ naturally formed an orthogonal sequence, with $r$ ranging over the non-negative integers, then $xy$ would be square summable. Unfortunately, $i_r\conj{y}$ forms an orthogonal sequence only for $\Sp_1=\mathbb{C}$, $\Sp_2=\mathbb{H}$ and $\Sp_3=\mathbb{O}$. 

\begin{theorem}
 If $n<4,$ $p,q\in G_n,$ $x\in\Sp_n$ and $0\ne p\ne q\ne 0,$  then \[ \ip{i_px}{i_qx}=0\]
\end{theorem}

 In $\Sp_4$ the sedendions, however, $\ip{i_2x}{i_5x}\ne0$ for general values of $x$.

It is not difficult to show that $x^2$ is square summable if $x$ is square summable.

\begin{definition}\label{D:norm}
 If $x\in\ell^2$ define the norm $\norm{x}=\sqrt{\ip{x}{x}}.$
\end{definition}

\begin{theorem}\label{T:zero2}
 If $r\ne0,$ then $\ip{x}{i_rx}=\ip{x}{xi_r}=0$
\end{theorem}
\begin{proof}
  Since $\cyd(r,r)=-1$ for $r\ne0$ and since $\cyd$ is proper, the result follows from Theorems \ref{T:zero} and \ref{T:zero1}.
 \end{proof}
\begin{theorem}\label{T:norm}
 If $x\in\ell^2,$ then $\norm{x}^2=x\conj{x}=\conj{x}x$
\end{theorem}
\begin{proof}
 $x\,\conj{x}=\sum_r\ip{x}{i_rx}i_r=\ip{x}{x}=\norm{x}^2$ and $\conj{x}x=\sum_r\ip{x\,i_r}{x}i_r=\ip{x}{x}=\norm{x}^2$ by Theorem \ref{T:zero2}.
\end{proof}
\begin{corollary}\label{C:inverse}
 If $x\in\ell^2,$ then $x^{-1}=\frac{\conj{x}}{\norm{x}^2}.$
\end{corollary}
\begin{theorem}\label{T:square}
 If $x\in\ell^2,$ then $x^2=2x_0x-\norm{x}^2.$
\end{theorem}
\begin{proof}
  Since $x+\conj{x}=2x_0,$ it follows that $x^2+x\conj{x}=2x_0x,$ thus
$x^2=2x_0x-x\conj{x}=2x_0x-\norm{x}^2.$
\end{proof}
\begin{corollary}
 If $x\in\ell^2$ and if $Re(x)=0$, then $x^2=-\norm{x}^2$ where $Re(x)=\frac{1}{2}\left(x+\conj{x}\right).$
\end{corollary}

\begin{theorem}\label{T:plc}
 If $x,y\in\ell^2,$ then
\[xy+yx=2\left(y_0x+x_0y\right)+\norm{x}^2+\norm{y}^2-\norm{x+y}^2\in\ell^2\]
\end{theorem}
\begin{proof}
 Since $xy+yx=(x+y)^2-x^2-y^2$ the result follows immediately from
Theorem \ref{T:square}.
\end{proof}
\begin{definition}\label{D:convolution}
 If $x,y\in\ell^2,$ define the \emph{dyadic convolution} of $x$ and $y$ as
  \[ x\ast y=\sum_r\left(\sum_p x_py_{pr}\right)i_r. \]
\end{definition}
\begin{remark}
 The convolution is simply the product which results from the \emph{trivial} twist $\iota(p,q)=1$ for all $p,q\in G.$
\end{remark}
\begin{conjecture}\label{Cj:convolution}
 If $x,y\in\ell^2,$ then $x\ast y\in\ell^2.$
\end{conjecture}
\begin{definition}\label{D:commutator}
 If $x,y\in\ell^2,$ define the \emph{commutator} of $x$ and $y$ as
  \[ [x,y]=xy-yx.\]
\end{definition}
\begin{theorem}\label{T:commutator}
 If $x,y\in\ell^2,$ then 
 \begin{align*} 
  [x,y] &= \sum_r\left(\sum_p [\cyd(p,pr)-\cyd(pr,p)]x_py_{pr}\right)i_r\\
        &= \sum_{r>1} \left[ \sum_{0 < p \ne r}\cyd(p,r)\left(x_{pr}y_p-x_py_{pr} \right) \right]i_r
 \end{align*}
\end{theorem}
\begin{proof}
 \[ [x,y]=xy-yx=\sum_r\left[ \ip{x}{i_r\conj{y}} - \ip{x}{\conj{y}i_r}\right]i_r \]
 \begin{align*}
  \ip{x}{i_r\conj{y}} &= \ip{\sum_px_pi_p}{i_r\sum_q\cyd(q,q)y_qi_q}\\
                      &= \ip{\sum_px_pi_p}{\sum_q\cyd(r,q)\cyd(q,q)y_qi_{rq}}\\
		      &= \ip{\sum_px_pi_p}{\sum_q\cyd(rq,q)y_qi_{rq}}\\
		      &= \ip{\sum_px_pi_p}{\sum_p\cyd(p,pr)y_{pr}i_p}\\
		      &= \sum_p\cyd(p,pr)x_py_{pr}
 \end{align*}
 
 \begin{align*}
  \ip{x}{\conj{y}i_r} &= \ip{\sum_px_pi_p}{\left(\sum_q\cyd(q,q)y_qi_q\right)i_r}\\
                      &= \ip{\sum_px_pi_p}{\sum_q\cyd(q,r)\cyd(q,q)y_qi_{qr}}\\
		      &= \ip{\sum_px_pi_p}{\sum_q\cyd(q,qr)y_qi_{rq}}\\
		      &= \ip{\sum_px_pi_p}{\sum_p\cyd(pr,p)y_{pr}i_p}\\
		      &= \sum_p\cyd(pr,p)x_py_{pr}
 \end{align*}
 Thus,
 \[  [x,y] = \sum_r\left(\sum_p [\cyd(p,pr)-\cyd(pr,p)]x_py_{pr}\right)i_r \]
Then 
\begin{align*}
  2[x,y] &= \sum_r\left\{
                  \left[
                        \sum_p
		   \left( 
		         \cyd(p,pr)-\cyd(pr,p)
                   \right)x_py_{pr}
	          \right]i_r
          +        \left[
	                \sum_q
		   \left(
		    \cyd(q,qr)-\cyd(qr,q)
	           \right)x_qy_{qr}
	          \right]i_r
	         \right\}\\
&=\sum_r\left\{
                  \left[
                        \sum_p
		   \left( 
		         \cyd(p,pr)-\cyd(pr,p)
                   \right)x_py_{pr}
	          \right]i_r
          +        \left[
	                \sum_p
		   \left(
		    \cyd(pr,p)-\cyd(p,pr)
	           \right)x_{pr}y_p
	          \right]i_r
	         \right\}\\
 &= \sum_r\left\{
           \sum_p\left[
	    \cyd(pr,p)-\cyd(p,pr)\right]
	         \left(
	    x_{pr}y_p-x_py_{pr}\right)
	   \right\}i_r\\
 &= \sum_r\left\{
           \sum_p\cyd(p,p)\left[
	    \cyd(r,p)-\cyd(p,r)\right]
	         \left(
	    x_{pr}y_p-x_py_{pr}\right)
	   \right\}i_r
 \end{align*}
 If $r=0$ or $p=0$ or $r=p$, then $\cyd(r,p)-\cyd(p,r)=0.$ If $p\ne0$ then $\cyd(p,p)=-1.$ And if $0\ne p\ne r\ne0$, then $\cyd(p,r)=-\cyd(r,p).$
 
 So,
 \[ 2[x,y]=\sum_{r>0}\left\{
           \sum_{0<p\ne r} 2\cyd(p,r)\left(
	    x_{pr}y_p-x_py_{pr}\right)
	   \right\}i_r.\]

Thus, 
\[ [x,y]=\sum_{r>0}\left\{
           \sum_{0<p\ne r} \cyd(p,r)\left(
	    x_{pr}y_p-x_py_{pr}\right)
	   \right\}i_r.\]
\end{proof}
\begin{corollary}
  $[x,y]$ is square summable if $x\ast y$ is.
\end{corollary}
\begin{corollary}
  $xy$ is square summable if $x\ast y$ is.
\end{corollary}

\section{Clifford Algebra}

In Clifford algebra, the same basis vectors $\mathcal{B}=\{i_p | p\in G\}$ will be used, as well as the same group $G$ of non-negative integers with group operation the bit-wise `exclusive or' operation. Only the twists will differ.

In Clifford algebra, the basis vectors are called `blades'. Each blade has a numerical `grade'.

$i_0=1$ is the unit scalar, and is a 0-blade.

$i_1,i_2,i_4,\cdots,i_{2^n}\cdot$ are 1-blades, or `vectors' in Clifford algebra parlance.

$i_3,i_5,i_6,\cdots$ are 2-blades or `bi-vectors'. The common characteristic of the subscripts is the fact that the sum of the bits of the binary representations of the subscripts is two.

$i_7,i_{11},i_{13},i_{14},\cdots$ are 3-blades or `tri-vectors', etc.

The grade of a blade equals the sum of the bits of its subscript.

As was the case with Cayley-Dickson algebras, this is not the standard notation. However, it has the advantage that the product of basis vectors satisfies $i_pi_q=\clf(p,q)i_{pq}$ for a suitably defined Clifford twist $\clf.$

In the standard notation, 1-blades or `vectors' are denoted $e_1,e_2,e_3,\cdots,$ whereas 2-blades or `bivectors' are denoted $e_{12}, e_{13}, e_{23}, \cdots$ etc.

Translating from the $e$-notation to the $i$-notation is straightforward. For example, the 3-blade $e_{134}$ translates as $i_{13}$ since the binary representation of 13 is 1101 with bits 1, 3 and 4 set. The 2-blade $e_{23}=i_6$ since the binary representation of 6 is 110, with bits 2 and 3 set.

Stated more formally, the ``$i$'' notation is related to the ``$e$'' notation in the following way: If $p_k\in\{0,1\}$ for $0\le k < n,$ and if $p=\sum_kp_k2^k$ then $i_p=\prod_ke_{(k+1)p_k}.$

There are four fundamental multiplication properties of 1-blades.
\begin{enumerate}
\item The square of 1-blades is 1. 
\item The product of 1-blades is anticommutative.
\item The product of 1-blades is associative.
\item Every $n$-blade can be factored into the product of $n$ distinct 1-blades.
\end{enumerate}

The convention is that, if $j<k,$ then $e_je_k=e_{jk},$ thus $e_ke_j=-e_{jk}.$

Any two $n$-blades may be multiplied by first factoring them into 1-blades. For example, the product of $e_{134}$ and $e_{23},$ is computed as follows:
\begin{align*} e_{134}e_{23}&=e_1e_3e_4e_2e_3\\
                            &=-e_1e_4e_3e_2e_3\\
			    &=e_1e_4e_2e_3e_3\\
			    &=e_1e_4e_2\\
			    &=-e_1e_2e_4\\
			    &=-e_{124}
\end{align*}

Since $e_{134}=i_{13}$ and $e_{23}=i_6,$ and the bit-wise `exclusive or' of 13 and 6 is 11, the same product using the `$i$' notation is
\[ i_{13}i_6=\clf(13,6)i_{11} \] so evidently, $\clf(13,6)=-1.$

As in the case of the Cayley-Dickson product, the $\clf$ function may be defined recursively.

Since the grade of a blade $i_p$ equals the sum of the bits of $p,$ it will be convenient to have a notation for the sum of the bits of a binary number. 

\begin{definition}
If $p$ is a binary number, let $\sob{p}$ denote the sum of the bits of $p.$ 
\end{definition}

The sum of the bits function can be defined recursively as follows:\begin{enumerate}
 \item $\sob{0}=0$
 \item $\sob{2p}=\sob{p}$
 \item $\sob{2p+1}=\sob{p}+1$
\end{enumerate}
\begin{lemma}
  \item $e_1i_{2p}=i_{2p+1}$
\end{lemma}
\begin{lemma}
  \item $e_1i_{2p+1}=i_{2p}$
\end{lemma}
\begin{lemma}
  \item $i_{2p}e_1=\sbf{p}i_{2p+1}$
\end{lemma}
\begin{lemma}
  \item $i_{2p+1}e_1=\sbf{p}i_{2p}$
\end{lemma}

\begin{theorem}\label{T:sign2}
There is a twist $\clf(p,q)$ mapping $G\times G$ into
$\{-1,1\}$ such that if $p,\,q\in G,$ then
$i_pi_q=\clf(p,q)i_{pq}.$
\end{theorem}
\begin{proof}
Let $G_n=\{p\:|\:0\le p < 2^n\}$ with group operation ``bit-wise exclusive or'' as in the Cayley-Dickson algebras.

We begin by noticing that $i_0i_0=\clf(0,0)i_0=1$ provided $\clf(0,0)=1.$

This defines the twist for $G_0.$

If $p$ and $q$ are in $G_{n+1},$ then there are elements $u$ and $v$ in $G_n$ such that one of the following is true:
\begin{enumerate}
 \item $p=2u$ and $q=2v$
 \item $p=2u$ and $q=2v+1$
 \item $p=2u+1$ and $q=2v$
 \item $p=2u+1$ and $q=2v+1$
\end{enumerate}
Assume $\clf$ is defined for $u,v\in G_n,$ then consider these four cases in order.

\begin{enumerate}
 \item $p=2u$ and $q=2v$\\
       \begin{align*}i_pi_q &= i_{2u}i_{2v}\\
                            &= \clf(u,v)i_{2uv}\\
                            &= \clf(2u,2v)i_{(2u)(2v)}\\
			    &= \clf(p,q)i_{pq}
       \end{align*} provided $\clf(2u,2v)=\clf(u,v).$
 \item $p=2u$ and $q=2v+1$
       \begin{align*}i_pi_q & =i_{2u}i_{2v+1}\\
                            & =i_{2u}e_1i_{2v}\\
                            & =\sbf{u}e_1i_{2u}i_{2v}\\
                            & =\sbf{u}e_1\clf(2u,2v)i_{2uv}\\
                            & =\sbf{u}\clf(u,v)i_{2uv+1}\\
 			    & =\clf(2u,2v+1)i_{2uv+1}\\
			    & =\clf(p,q)i_{pq}
       \end{align*} provided $\clf(2u,2v+1)=\sbf{u}\clf(u,v).$
 \item $p=2u+1$ and $q=2v$
       \begin{align*} i_pi_q &=i_{2u+1}i_{2v}\\
                             &=e_1i_{2u}i_{2v}\\
			     &=e_1\clf(u,v)i_{2uv}\\
			     &=\clf(u,v)i_{2uv+1}\\
			     &=\clf(2u+1,v)i_{2uv+1}\\
			     &=\clf(p,q)i_{pq}      
       \end{align*} provided $\clf(2u+1,2v)=\clf(u,v).$
 \item $p=2u+1$ and $q=2v+1$
       \begin{align*} i_pi_q &=i_{2u+1}i_{2v+1}\\
                             &=e_1i_{2u}e_1i_{2v}\\
                             &=\sbf{u}e_1e_1i_{2u}i_{2v}\\
			     &=\sbf{u}\clf(u,v)i_{2uv}\\
			     &=\clf(2u+1,2v+1)i_{2uv}\\
			     &=\clf(p,q)i_{pq}
       \end{align*} provided $\clf(2u+1,2v+1)=\sbf{u}\clf(u,v).$
\end{enumerate}
\end{proof}

\begin{corollary}
 Assume $p,q\in G_n.$ The Clifford algebra twist can be defined recursively as follows:
 \begin{enumerate}
  \item $\clf(0,0)=1$
  \item $\clf(2p,2q)=\clf(2p+1,2q)=\clf(p,q)$
  \item $\clf(2p,2q+1)=\clf(2p+1,2q+1)=\sbf{p}\clf(p,q)$
 \end{enumerate}
\end{corollary}

\begin{lemma}
 $\sbf{u}\sbf{v}=\sbf{uv}$
\end{lemma}
\begin{proof}
 This follows from the fact that $\sob{u} + \sob{v} = \sob{uv} +2\sob{u\,\wedge\,v}$ where $u\,\wedge\,v$ represents the bitwise `and' of $u$ and $v.$
\end{proof}
\begin{theorem}
 The Clifford twist is associative.
\end{theorem}
\begin{proof}
 By Definition \ref{D:associative}, the twist $\clf: G \times G \mapsto \{ -1,1\}$ is associative provided $\clf(p,q)\clf(pq,r)=\clf(p,qr)\clf(q,r)$ for $p,q,r\in G.$ This is true for $G_0$ since $\clf(0,0)=1.$
 
 Suppose $\clf(u,v)\clf(uv,w)=\clf(u,vw)\clf(v,w)$ for $u,v,w\in G_n.$ Let $p,q,r$ be in $G_{n+1}.$ Then there are $u,v,w\in G_n$ such
that one of the following eight cases is true.
 \begin{enumerate}
  \item $p=2u,q=2v,r=2w$  
  \begin{align*}\text{Then\ }\clf(p,q)\clf(pq,r)&=\clf(2u,2v)\clf(2uv,2w)\\
				                &=\clf(u,v)\clf(uv,w)\\
					        &=\clf(u,vw)\clf(v,w)\\
				                &=\clf(2u,2vw)\clf(2v,2w)\\
						&=\clf(p,qr)\clf(q,r)
  \end{align*}
  \item $p=2u,\;q=2v,\;r=2w+1$
  \begin{align*}\text{Then\ }\clf(p,q)\clf(pq,r)&=\clf(2u,2v)\clf(2uv,2w+1)\\
                                                &=\clf(u,v)\sbf{uv}\clf(uv,w)\\
						&=\sbf{uv}\clf(u,vw)\clf(v,w)\\
						&=\sbf{u}\clf(u,vw)\sbf{v}\clf(v,w)\\
						&=\clf(2u,2vw+1)\clf(2v,2w+1)\\
				                &=\clf(p,qr)\clf(q,r)
  \end{align*}
  \item $p=2u,\;q=2v+1,\;r=2w$
  \begin{align*}\text{Then\ }\clf(p,q)\clf(pq,r)&=\clf(2u,2v+1)\clf(2uv+1,2w)\\
                                                &=\sbf{u}\clf(u,v)\sbf{uv}\clf(uv,w)\\
						&=\sbf{u}\sbf{u}\clf(u,vw)\sbf{v}\clf(v,w)\\
						&=\clf(2u,2vw+1)\clf(2v+1,2w)\\
				                &=\clf(p,qr)\clf(q,r)
  \end{align*}
  \item $p=2u,\;q=2v+1,\;r=2w+1$
  \begin{align*}\text{Then\ }\clf(p,q)\clf(pq,r)&=\clf(2u,2v+1)\clf(2uv+1,2w+1)\\
						&=\sbf{u}\clf(u,v)\sbf{uv}\clf(uv,w)\\
						&=\sbf{u}\sbf{u}\clf(u,vw)\sbf{v}\clf(v,w)\\
						&=\clf(2u,2vw)\clf(2v+1,2w+1)\\
				                &=\clf(p,qr)\clf(q,r)
  \end{align*}
  \item $p=2u+1,\;q=2v,\;r=2w$
  \begin{align*}\text{Then\ }\clf(p,q)\clf(pq,r)&=\clf(2u+1,2v)\clf(2uv+1,2w)\\
                                                &=\clf(u,v)\clf(uv,w)\\
						&=\clf(u,vw)\clf(v,w)\\
						&=\clf(2u+1,2vw)\clf(2v,2w)\\
				                &=\clf(p,qr)\clf(q,r)
  \end{align*}
  \item $p=2u+1,\;q=2v,\;r=2w+1$
  \begin{align*}\text{Then\ }\clf(p,q)\clf(pq,r)&=\clf(2u+1,2v)\clf(2uv+1,2w+1)\\
                                                &=\clf(u,v)\sbf{vw}\clf(uv,w)\\
						&=\sbf{v}\clf(u,vw)\sbf{w}\clf(v,w)\\
						&=\clf(2u+1,2vw+1)\clf(2v,2w+1)\\
				                &=\clf(p,qr)\clf(q,r)
  \end{align*}
  \item $p=2u+1,\;q=2v+1,\;r=2w$
  \begin{align*}\text{Then\ }\clf(p,q)\clf(pq,r)&=\clf(2u+1,2v+1)\clf(2uv,2w)\\
                                                &=\sbf{u}\clf(u,v)\clf(uv,w)\\
						&=\sbf{u}\clf(u,vw)\clf(v,w)\\
						&=\clf(2u+1,2vw+1)\clf(2v+1,2w)\\
				                &=\clf(p,qr)\clf(q,r)
  \end{align*}
  \item $p=2u+1,\;q=2v+1,\;r=2w+1$
  \begin{align*}\text{Then\ }\clf(p,q)\clf(pq,r)&=\clf(2u+1,2v+1)\clf(2uv,2w+1)\\
                                                &=\sbf{u}\clf(u,v)\sbf{uv}\clf(uv,w)\\
						&=\sbf{u}\sbf{u}\sbf{v}\clf(u,vw)\clf(v,w)\\
						&=\clf(u,vw)\sbf{v}\clf(v,w)\\
						&=\clf(2u+1,2vw)\clf(2v+1,2w+1)\\
				                &=\clf(p,qr)\clf(q,r)
  \end{align*}
 \end{enumerate}
 Thus $\clf$ is associative for $G_{n+1}$ if it is associative for $G_n$ establishing the associativity of $\clf$ for all $G_k.$
\end{proof}

By Theorem \ref{T:assocprod}, $\clf$ is proper, thus, if $x,y\in\mathcal{C}_n,$ where $\mathcal{C}_n$ is the $n$-dimensional Clifford algebra, then by Theorem \ref{T:product} the product $xy$ may be computed from the sum

				   \begin{align*}
                                     xy &= \sum_{r\in G_n}\ip{x}{i_r\conj{y}}i_r\\
                                        &= \sum_{r\in G_n}\ip{\conj{x}i_r}{y}i_r
                                   \end{align*}

where the conjugate of a multivector $x$ is $\conj{x}=\sum_{p\in G_n}\clf(p,p)x_pi_p$ by Theorem \ref{T:conjugate} part $i.$

\begin{theorem}
 If $p\in G_n,$ then $\clf(p,p)=(-1)^s,$ where $s$ is the triangular number $T_\sob{p}=\frac{\sob{p}\left[\sob{p}-1\right]}{2}.$
\end{theorem}

\begin{proof}
 The proof is by induction. It is true for $G_0,$ since if $p=0$ then $\sob{p}=0$ so $s=0,$ thus $\clf(p,p)=1=(-1)^s.$
 
 Assume the relation is true for $G_n.$ Let $p\in G_{n+1}.$ Then there is a $u\in G_n$ such that either $p=2u$ or $p=2u+1.$
 
Suppose $p=2u.$ Then $u\in G_n,$ so $\clf(u,u)=(-1)^s,$ where $s=\frac{\sob{u}\left[\sob{u}-1\right]}{2}.$ But $\sob{u}=\sob{2u}$ and $\clf(u,u)=\clf(2u,2u)=\clf(p,p).$ So $\clf(p,p)=\frac{\sob{p}\left[\sob{p}-1\right]}{2}.$
 
Suppose $p=2u+1.$ Then 
\begin{align*}
   \clf(p,p)&=\clf(2u+1,2u+1)\\
            &=\sbf{u}\clf(u,u)\\
	    &=\sbf{u}(-1)^t \text{\ where $t=\frac{\sob{u}\left[\sob{u}-1\right]}{2}$}\\
	    &=(-1)^s \text{\ where $s=\sob{u}+\frac{\sob{u}\left[\sob{u}-1\right]}{2}$}
\end{align*}
 Thus, \begin{align*}
         s&=\frac{\left[\sob{u}+1\right]\sob{u}}{2}\\
	  &=\frac{\sob{p}\left[\sob{p}-1\right]}{2}
       \end{align*}

Thus for $p\in G_{n+1},$ $\clf(p,p)=(-1)^s,$ where $s=\frac{\sob{p}\left[\sob{p}-1\right]}{2}.$
\end{proof}
\begin{corollary}
 If $x\in\mathcal{C}_n$, $p\in G_n$ and if $T_\sob{p}$ is an odd triangular number, then $\ip{x}{i_px}=0$.
\end{corollary}

\begin{theorem}
 $\sgn(q,p)=\sgn(p,q)(-1)^{\sob{p}\sob{q}-\sob{p\wedge q}}$
\end{theorem}
\begin{proof}
 The multiblades $i_p$ and $i_q$ contain exactly $\sob{p\wedge q}$ 1-blade factors in common. Assume $p\wedge q=0$. Factor $i_p$ and $i_q$ into the product of 1-blades. To get from the arrangement $i_pi_q$ to the arrangement $i_qi_p,$ swap the left-most 1-blade factor of $i_q$ with each of the 1-blade factors of $i_p$ resulting in $\sob{p}$ changes of sign. Repeat this for each of the $\sob{q}$ 1-blade factors of $i_q$ from left to right until each 1-blade of $i_q$ has been swapped with each of the $\sob{p}$ 1-blades of $i_p.$ This results in $\sob{p}\sob{q}$ changes of sign to reverse $i_p$ and $i_q.$ If $p\wedge q\ne0,$ then the sign will fail to change for a total of $\sob{p\wedge q}$ times. Thus, in general, there will be $\sob{p}\sob{q}-\sob{p\wedge q}$ changes of sign in reversing the product of $i_p$ and $i_q.$
\end{proof}

It is easily verified that $(-1)^{\sob{p}\sob{q}}$ and $(-1)^\sob{p\wedge q}$ are associative twists on $G.$ The latter of these two is the \emph{Hadamaard} twist. Interestingly, this is the twist which results if, instead of the Cayley-Dickson construction, one defines the product of an ordered pair as
\[ (a,b)(c,d)=(ac-bd,ad+bc) \]
Another interesting twist which is not associative, but proper is $(-1)^\sob{pq}.$
\section{Conclusion}

 My primary purpose in writing this paper was to document some of my notes.
 
 There are many questions and avenues of further research in the area of proper twists on groups and their resulting twisted algebras. 
 
 The question whether the Hilbert space $\ell^2$ is a Cayley-Dickson algebra is open, so far as I know, and would be resolved by a proof of Conjecture \ref{Cj:convolution}. In fact, a resolution of this conjecture in the affirmative would render $\ell^2$ closed under any proper twist defined on the group $G$ of non-negative integers under the ``exclusive or'' product.
 
 The set $\mathfrak{S}_P(G)$ of proper twists on an arbitrary group G is an abelian group. The set $\mathfrak{S}_A(G)$ of associative twists is a subgroup of $\mathfrak{S}_P(G)$. Do these groups have any interesting properties?

\end{document}